\documentclass[11pt]{amsart}

\usepackage[english]{babel}
\usepackage{amsmath}
\usepackage{amsfonts}
\usepackage{amssymb}
\usepackage{amsthm}
\usepackage{epsfig}
\usepackage{graphicx}
\usepackage{url,hyperref}

\newtheorem{theorem}{Theorem}[section]
\newtheorem{lemma}[theorem]{Lemma}
\newtheorem{corollary}[theorem]{Corollary}
\newtheorem{proposition}[theorem]{Proposition}

\theoremstyle{definition}

\newtheorem{definition}[theorem]{Definition}
\newtheorem{remark}[theorem]{Remark}

\theoremstyle{definition}

\def\cl{\mathop\mathrm{cl}\nolimits}
\def\B{\mathrm{B}}
\def\s{\mathbb{S}}
\def\K{\mathcal{K}}
\def\r{\mathcal{R}}
\def\U{\mathcal{U}}
\def\R{\mathbb{R}}

\def\vol{\mathrm{vol}}
\def\S{\mathrm{S}}
\def\inr{\mathrm{r}}
\def\W{\mathrm{W}}
\def\abs#1{\left| #1\right|}
\newcommand{\form}[1]{#1^{*}}

\numberwithin{equation}{section}

\title{Isoperimetric relations for inner parallel bodies}

\author{M. A. Hern\'andez Cifre}
\address{Departamento de Matem\'aticas, Universidad de Murcia, Campus de
Espinar\-do, 30100-Murcia, Spain} \email{mhcifre@um.es}

\author{E. Saor\'\i n G\'omez}
\address{ALTA institute for Algebra, Geometry, Topology and their
Applications, Universit\"at Bremen, 28334-Bremen, Germany}
\email{esaoring@uni-bremen.de}

\thanks{This research is a result of the activity developed within the
framework of the Programme in Support of Excellence Groups of the Regi\'on
de Murcia, Spain, by Fundaci\'on S\'eneca, Science and Technology Agency
of the Regi\'on de Murcia. The work is partially supported by MICINN/FEDER
project PGC2018-097046-B-I00 and Fundaci\'on S\'eneca project
19901/GERM/15}



\begin{document}

\begin{abstract}
We analyze aspects of the behavior of the family of inner parallel bodies
of a convex body for the isoperimetric quotient and deficit of arbitrary
quermassintegrals. By means of technical boundary properties of the
so-called form body of a convex body and similar constructions for inner
parallel bodies, we point out an erroneous use of a relation between the
latter bodies in two different works. We correct these results, limiting
them to convex bodies having a very precise boundary structure.
\end{abstract}

\maketitle

\section{Introduction}

Inner parallel bodies of convex bodies have been object of recent studies
with different flavors
\cite{Domokos,HCS,HCS2,HCS3,HCS_rbeta,Larson,Linke}. More classical
existing literature on them (e.g. \cite{Bol43,Di,Di49,Ha55,Ha57,SYphd})
along with its role in the proofs of fundamental results in the theory of
convex bodies, make inner parallel bodies an essential object not only
within classical Convex Geometry (see \cite[Section~7.5]{Sch}), but also
in other related fields (see e.g. \cite{GK,HW,RV} and the references in
\cite[Note~3 for Section~3.1]{Sch} and \cite{Domokos}).

In \cite{Domokos} and \cite{Larson} the authors study the behavior of the
isoperimetric quotient for the family of inner parallel bodies, and
provide a lower bound for the perimeter of the inner parallel bodies of a
convex body, respectively. However, in both articles they happen to make
an erroneous use of the relation $K\subseteq
K_{\lambda}+\abs{K}K_{\lambda}^*$ between the inner parallel bodies
$K_{\lambda}$ of a convex body, their form bodies $K_{\lambda}^*$ and the
original convex body $K$ (see \eqref{e: eq contraej} and
Section~\ref{sect:backgrnd} for the proper definitions). This relation,
which holds, for example, under technical properties of the boundary of
the involved convex bodies (see \eqref{e: cond extremos SY}), is, however,
not true without further conditions. To the best of the authors'
knowledge, a full characterization of the conditions under which the above
inclusion holds is not known.

The purpose of this paper is two-sided. On the one hand, we describe the
error contained in the two mentioned references, providing with examples
proving these have to be adjusted with further hypotheses in order to
hold. On the other hand, we provide alternative proofs to those results
under suitable restrictions of the boundaries of the involved convex
bodies, and further, we extend the results concerning inner parallel
bodies in \cite{Domokos} to a more general setting.

The paper is organized as follows. In Section \ref{sect:backgrnd} we
introduce the notions and basic results, which are needed throughout the
paper. In Section \ref{s: counterexamples} we analyze the problems in the
proof of the main result in \cite{Domokos}, providing an example where the
used methods do not hold. In Section \ref{s: quotient and deficit} we
obtain new results concerning the behavior of the isoperimetric quotient
and deficit under assumptions on the boundary of the involved convex
bodies. Finally in Section \ref{s: perimeter} we point out an error -of
the same spirit of the one found in \cite{Domokos}- in one of the proofs
of \cite{Larson} and discuss it.

\section{Background}\label{sect:backgrnd}

Let $\K^n$ be the set of all convex bodies, i.e., nonempty compact convex
subsets of the Euclidean space $\R^n$, and let $\K^n_n$ the subset of
convex bodies having interior points. A convex body $K$ is called {\em
regular} if all its boundary points are regular, i.e., the supporting
hyperplane to $K$ at any boundary point is unique. Let $\B_n$ be the
$n$-dimensional Euclidean unit ball and $\s^{n-1}$ the corresponding unit
sphere. The volume of a measurable set $M\subset\R^n$, i.e., its
$n$-dimensional Lebesgue measure, is denoted by $\vol(M)$, and the measure
of its boundary, i.e., its surface area (also called perimeter), is
represented by $\S(M)$. Furthermore, the closure of $M$ is denoted by $\cl
M$. For $K\in\K^n$ and $u\in\s^{n-1}$, $h(K,u)=\sup\bigl\{\langle
x,u\rangle:x\in K\bigr\}$ stands for the {\em support function} of $K$
(see e.g. \cite[Section 1.7]{Sch}).

The vectorial or Minkowski addition of two sets $K,L\subset \R^n$ is given
by
\[
K+L=\{x+y:x\in K,\,y\in L\},
\]
whereas the Minkowski difference of $K,L\subset \R^n$ is given by
\[
K\sim L=\{x\in\R^n\,:\, x+L\subseteq K\}.
\]
We notice that $(K\sim L)+L\subseteq K$, and the inequality may be strict.

Let $K\in\K^n$ and $E\in\K^n_n$. The {\em inradius} $\inr(K;E)$ of $K$
relative to $E$ is the radius of one of the largest dilations of $E$ which
fits inside $K$, i.e.,
\[
\inr(K;E)=\sup\{r\geq 0:\exists\,x\in\R^n\text{ with } x+rE\subseteq K\}.
\]
For $-\inr(K;E)\leq\lambda\leq 0$ the {\em inner parallel body of $K$ at
distance $\abs{\lambda}$} is the Minkowski difference of $K$ and
$\abs{\lambda}E$, i.e.,
\[
K_{\lambda}:=K\sim\abs{\lambda}E=\bigl\{x\in\R^n:x+\abs{\lambda}E\subseteq
K\bigr\}\in\K^n.
\]
Notice that if $E=B_n$, then $K_{-\inr(K;\B_n)}$ is the set of incenters
of $K$, which is usually called the {\em kernel} of $K$, and its dimension
is strictly less than $n$ (see \cite[p.~59]{BF}). Equivalently (see
\cite[Section 3.1]{Sch}), the inner parallel body $K_{\lambda}$ of $K$,
$-\inr(K;E)\leq\lambda\leq 0$, can be defined using the support functions
of $K$ and $E$ as
\begin{equation}\label{e:K_lambda1}
K_{\lambda}=\bigl\{x\in\R^n: \langle x,u\rangle \leq h(K,u)-\abs{\lambda}
h(E,u), u\in\s^{n-1}\bigr\}.
\end{equation}

A vector $u\in\s^{n-1}$ is a {\em $0$-extreme normal vector} (or just {\em
extreme vector}) of $K$ if it cannot be written as a linear combination of
two linearly independent normal vectors at one and the same boundary point
of $K$. We denote by $\U(K)$ the set of $0$-extreme normal vectors of $K$,
which play a key role in the study of convex bodies. For instance, $\U(K)$
is the smallest set one can use so that
\begin{equation}\label{e:K_cap_U(K)}
K=\bigl\{x\in\R^n:\langle x,u\rangle\leq h(K,u), u\in\U(K)\bigr\}
\end{equation}
(see e.g. \cite[Corollary 1.4.5 or page~386]{Sch}), and thus, the inner
parallel bodies of $K$ can be expressed as (cf. \eqref{e:K_lambda1})
\begin{equation*}
K_{\lambda}=\bigl\{x\in\R^n:\langle x,u\rangle\leq
h(K,u)-\abs{\lambda}h(E,u),u\in\U(K)\bigr\}
\end{equation*}
for $-\inr(K;E)\leq\lambda\leq 0$.

The (relative) {\em form body} of a convex body $K\in\K^n_n$ {\em with
respect to} $E\in\K^n_n$, denoted by $K^*$, is defined as (see e.g.
\cite{Di})
\begin{equation*}
\form{K}=\bigl\{x\in\R^n:\langle x,u\rangle\leq h(E,u),u\in\U(K)\bigr\}.
\end{equation*}
We notice that $K^*$ strongly depends on the body $E$. Nevertheless, and
for the sake of simplicity, we omit $E$ in the notation.

The form body of $K\in\K^n_n$ (with respect to an arbitrary $E\in\K^n_n)$
is always a tangential body of $E$. We recall that a convex body
$K\in\K^n$ containing a convex body $E\in\K^n$, is called a {\em
tangential body} of $E$, if through each boundary point of $K$ there
exists a support hyperplane to $K$ that also supports $E$. We notice that
if $K$ is a tangential body of $E$, then $\inr(K;E)=1$.

There is also a very close connection between inner parallel bodies and
tangential bodies. The next result enlighten it.
\begin{theorem}{\cite[Lemma 3.1.14]{Sch}}\label{t:Schneider}
Let $K,E\in\K^n_n$ and let $-\inr(K;E)<\lambda<0$. Then $K_{\lambda}$ is
homothetic to $K$ if and only if $K$ is homothetic to a tangential body of
$E$.
\end{theorem}

\begin{remark}\label{r:tangential}
The proof of Theorem \ref{t:Schneider} shows that if $K$ is a tangential
body of $E$ then $K_{\lambda}=(1+\lambda)K$ for $-1<\lambda\leq 0$.
\end{remark}

In the following, we collect some standard properties of inner parallel
bodies, form bodies and extreme vectors, together with other relations
through the Minkowski sum, which will be needed later on. There exist
further relations, in a stronger form, through the so-called
Riemann-Minkowski integral, for which we refer to \cite{Dinghas56} and
\cite[Lemma 3.2]{SYphd}.

\begin{lemma}\label{Collection of properties}
Let $K,L\in\K^n$ and $E\in\K^n_n$. The following properties hold:
\begin{enumerate}\itemsep1pt
\item $\U(K_{\lambda})\subseteq \U(K)$ for $-\inr(K;E)<\lambda\leq 0$ (see
\cite[Lemma 4.5]{SYphd}).
\item If $K\in\K^n_n$ and $E$ is regular then $\cl\U(K)=\U(\form{K})$ (see
\cite[Lem\-ma~2.6]{SYphd} and \cite[Lemma 2.1]{HCS}).
\item $\U(K)\cup\U(L)\subseteq\U(K+L)=\U(K+\mu L)$ for $\mu>0$. The inclusion
may be strict (see \cite[Lemma 2.4]{SYphd} and \cite[Lemma 3.1]{HCS3}).
\item $K_{\lambda}+\abs{\lambda}E\subseteq K$ for $-\inr(K;E)\leq\lambda\leq 0$ (see \cite[(4.1)]{SYphd}).
\item If $K\in\K^n_n$ then $K_{\lambda}+\abs{\lambda}K^*\subseteq K$ for $-\inr(K;E)\leq\lambda\leq 0$ (see
\cite[Lemma 4.8]{SYphd}).
\end{enumerate}
\end{lemma}

\begin{remark}\label{r: equalities}
The equality cases in Lemma \ref{Collection of properties} (iv) and (v)
are well-known:
\begin{enumerate}
\item Equality holds in (iv) for all $-\inr(K;E)\leq\lambda\leq 0$ if and only
if $K=K_{-\inr(K;E)}+\inr(K;E)E$ (see \cite[p. 81]{SYphd}).

\item If $E$ is regular, equality holds in (v) for all
$-\inr(K;E)\leq\lambda\leq 0$ if and only if $K$ is a tangential body of
$K_{-\inr(K;E)}+\inr(K;E)E$ satisfying $\U(K)=\U(K_{\lambda}+K^*)$ for all
$-\inr(K;E)\leq\lambda\leq 0$ (see \cite[Theorem~2.2]{HCS3}).
\end{enumerate}
\end{remark}

Let $K,E\in\K^n_n$. From now on we will write
$K_{\lambda}^*=(K_{\lambda})^*$ to denote the form body of the inner
parallel body of $K$ at distance $\abs{\lambda}$, $-\inr(K;E)<\lambda\leq
0$. The following counterpart of the relations contained in Lemma
\ref{Collection of properties} (v), can be found in \cite[Corollary to
Lemma 4.8]{SYphd} (see also Lemma~\ref{Collection of properties} (ii)).

\begin{proposition}[{\cite[Corollary to
Lemma 4.8]{SYphd}}]\label{p: cond extremos SY} Let $K,E\in\K^n_n$, with
$E$ regular. Assume that, for some $-\inr(K;E)<\lambda<0$, the relation
\begin{equation}\label{e: cond extremos SY}
\U(K_\lambda^*)=\U(K_\lambda+K_\lambda^*)
\end{equation}
holds. Then,
\begin{equation}\label{e: eq contraej}
K\subseteq K_\lambda+\abs{\lambda}K_\lambda^*.
\end{equation}
For $n=2$ there is equality in \eqref{e: eq contraej} for all
$K\in\K^2_2$.
\end{proposition}

Condition \eqref{e: cond extremos SY} deserves further observations. On
the one hand, it is similar to the identity $\U(K_\lambda+K^*)=\U(K^*)$,
which is a direct consequence of the relation $\U(K_\lambda+K^*)=\U(K)$
needed in \cite[Theorem~2.2]{HCS3} (see Remark~\ref{r: equalities}),
together with Lemma~\ref{Collection of properties} (ii) and (iii).
However, since examples of convex bodies for which
$\U(K_\lambda)\subset\U(K)$ (strictly) are easily constructed, both
conditions are different. On the other hand, Lemma~\ref{Collection of
properties} (iii) yields $\U(K_\lambda^*)\subseteq
\U(K_\lambda+K_\lambda^*)$; however, the inclusion may be strict (see
Section \ref{s: counterexamples}).

\section{Convex bodies not satisfying the inclusion $K\subseteq K_\lambda+\abs{\lambda}K_\lambda^*$}\label{s: counterexamples}

Let $K\in\K^n$ be a convex body. In \cite{Domokos} the authors study the
isoperimetric quotient
$I(K_{-\lambda}):=\vol(K_{-\lambda})/\S(K_{-\lambda})^{n/(n-1)}$ of the
family of inner parallel bodies $K_{-\lambda}$, $0\leq
\lambda<\inr(K;\B_n)$, when $E=\B_n$, and analyze the behavior of the
function $I(\lambda)=I(K_{-\lambda})$: in Theorem~1 they prove that the
isoperimetric quotient function $I(\lambda)$ is non-increasing in $0\leq
\lambda<\inr(K;\B_n)$ {\em for all convex bodies}. However, as we
mentioned in the introduction, the proof of this result is erroneous.

The main idea of the proof is to bound from below the quotient defining
$I(\lambda)$. To this end, the numerator $\vol(K_\lambda)$ is bounded
using Lemma \ref{Collection of properties} (iv) and the property
$(K_{-\lambda})_{-\mu}=K_{-\lambda-\mu}$ for
$0\leq\lambda,\mu\leq\lambda+\mu<\inr(K;\B_n)$ (see \cite[(3.17)]{Sch}).
More precisely, and following the notation in \cite{Domokos}, for $0\leq
t\leq t_0<\inr(K;\B_n)$,
\[
\vol(K_{-t})\geq\vol\bigl(K_{-t_0}+\abs{t-t_0}\B_n\bigr).
\]
In order to bound the denominator $\S(K_{-t})$, the authors make use of
the monotonicity of the surface area applied to the content \eqref{e: eq
contraej}, namely,
\[
K_{-t}\subseteq K_{-t_0}+\abs{t-t_0}\form{K_{-t_0}}
\]
for $0\leq t\leq t_0<\inr(K;\B_n)$. However, this inclusion is not true
without further conditions (as, for instance, the equality
$\U(K_{-t_0}^*)=\U(K_{-t_0}+K_{-t_0}^*)$, see Proposition~\ref{p: cond
extremos SY}). Indeed, in the next, we prove that the content $K\subseteq
K_\lambda+\abs{\lambda}K_\lambda^*$, $-\inr(K;E)<\lambda\leq 0$, is not
valid in its full generality.

For $K\in\K^n_n$ and $\mu\geq 0$, we consider the following convex body:
\begin{equation}\label{e: family counterexample}
K(\mu):=\bigl\{x\in\R^n:\langle x,u\rangle\leq h(K,u)+\mu h(E,u),
u\in\U(K)\bigr\} .
\end{equation}
This construction appeared already in \cite{Larson, Linke,SYphd}. Indeed,
in \cite{SYphd} the following result was proven.

\begin{proposition}\label{p: construction counterexample}\cite{SYphd}
Let $K,E\in\K^n_n$ and let $\mu\geq 0$. Then
\begin{enumerate}\itemsep1pt
\item $K+\mu E\subseteq K+ \mu K^*\subseteq K(\mu)$.
\item $\inr\bigl(K(\mu);E\bigr)=\mu+\inr(K;E)$.
\item For $-\inr(K;E)-\mu\leq \lambda\leq 0$ we have
\[
K(\mu)_{\lambda}=\left\{\begin{array}{ll}
K(\mu+\lambda) & \;\text{ for }-\mu\leq\lambda\leq 0,\\[1mm]
K_{\lambda+\mu} & \;\text{ for }-\inr(K;E)-\mu\leq\lambda\leq -\mu.
\end{array}\right.
\]
%\item $\mu\form{K}\subseteq K(\mu)$.
\end{enumerate}
\end{proposition}

We will also need the following additional result.
\begin{lemma}\label{l:U(K)<U(Kmu)=clU(K)}
Let $K\in\K^n$ and let $E\in\K^n_n$ be regular. Then, for any $\mu\geq 0$,
\begin{equation}
\U\bigl(K(\mu)\bigr)=\cl\U(K).
\end{equation}
\end{lemma}

\begin{proof}
It is enough to observe that, from the definition of $K(\mu)$, it follows
that $K(\mu)$ is the form body of $K$ with respect to $E'=K+\mu E$. Since
$E$ is regular, so is $K+\mu E=E'$, and hence the identity follows from
Lemma~\ref{Collection of properties}~(ii).
\end{proof}

%\begin{lemma}\cite[Lemma 6.1]{Linke}\label{l: P(mu) extremos}
%Let $P\in\K^n_n$ be a polytope and let $\mu\geq 0$. Then
%$\U(P)=\U\bigl(P(\mu)\bigr)$.
%\end{lemma}

We are now in a position to prove the announced non-validity of the
inclusion \eqref{e: eq contraej} without further assumptions. Next result
will be proven when the body $E=\B_n$, although it holds true for any
regular $E\in\K^n_n$.

\begin{proposition}\label{error}
There exists $K\in\K^n_n$ such that $K\supset
K_\lambda+\abs{\lambda}K_\lambda^*$ strictly for some
$-\inr(K;\B_n)<\lambda<0$.
\end{proposition}

\begin{proof}
For any $K\in\K^n_n$ and $\mu\geq 0$, since $K=K(\mu)_{-\mu}$ is an inner
parallel body of $K(\mu)$ (see Proposition~\ref{p: construction
counterexample} (iii)), item (i) of Proposition \ref{p: construction
counterexample} yields
\begin{equation}\label{e: contenidos contraej}
K(\mu)_{-\mu}+\mu\form{K(\mu)_{-\mu}}=K+\mu\form{K}\subseteq K(\mu).
\end{equation}
So, we have to find a convex body $K$ (and $\mu>0$) such that the above
inclusion is strict.

If we assume, to the contrary, that
$K(\mu)_{-\mu}+\mu\form{K(\mu)_{-\mu}}=K(\mu)$ for all $K\in\K^n_n$ and
$\mu\geq 0$, then we have, in particular, that $K(\mu)=K+\mu\form{K}$.
Then, since $\U(K^*)=\cl\U(K)\supseteq\U(K)$ (Lemma~\ref{Collection of
properties} (ii)), we can use Lemma~\ref{Collection of properties}~(iii)
and Lemma~\ref{l:U(K)<U(Kmu)=clU(K)} to get
\[
\U(K^*)=\U(K)\cup\U(K^*)\subseteq\U(K+K^*)=\U\bigl(K(\mu)\bigr)=\cl\U(K)=\U(K^*).
\]
Hence
\begin{equation}\label{e:U(K*)=U(K+K*)}
\U(\form{K})=\U(K+\form{K}).
\end{equation}

Now, it will be enough to find a convex body for which the latter equality
does not hold, and so we will get the desired contradiction. The following
polytope $P$ makes the job (see Figure~\ref{f:P(12)}; we notice that it
coincides with the polytope $P(12)$ used in
\cite[Proposition~5.1]{Linke}). Let
\begin{equation}\label{e: p}
P=\left\{\begin{pmatrix}x_1\\x_2\\x_3\end{pmatrix}\in\R^3:
\begin{array}[c]{rcl}
\pm 12x_1+35x_3&\leq&432,\\
\pm 12x_2+5x_3&\leq&60,\\
x_3&\geq&0\\
\end{array}\right\}.
\end{equation}

\begin{figure}[h]
\begin{center}
\includegraphics[width=6.5cm]{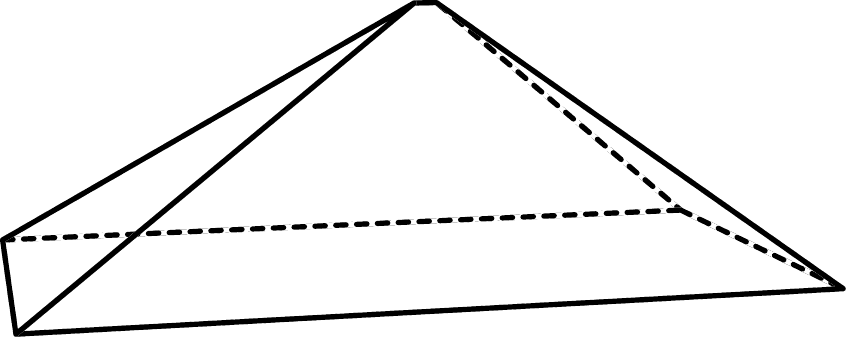}
\caption{A polytope such that $P(\mu)\supset P(\mu)_{-\mu}+\mu
P(\mu)_{-\mu}^*$ strictly, $\mu>0$.}\label{f:P(12)}
\end{center}
\end{figure}

On the one hand, since
\begin{equation*}
\form{P}=\left\{\begin{pmatrix}x_1\\x_2\\x_3\end{pmatrix}\in\R^3:
\begin{array}[c]{rcl}
\pm 12x_1+35x_3&\leq&37,\\
\pm 12x_2+5x_3&\leq&13,\\
x_3&\geq&-1\\
\end{array}\right\}
\end{equation*}
then
\[
\U(P^*)=\left\{\left(\pm
\frac{12}{37},0,\frac{35}{37}\right)^{\intercal},\left(0,\pm
\frac{12}{13},\frac{5}{13}\right)^{\intercal},(0,0,-1)^{\intercal}\right\}.
\]
On the other hand, the polytope $P+\form{P}$ has a facet with unit outer
normal vector $(0,0,1)^{\intercal}\not\in\U(\form{P})$, which arises from
the edge of $P$ determined by the straight line $\{x_2=0,x_3=12\}$, and
the edge of $\form{P}$ corresponding to the line $\{x_1=0,x_3=37/35\}$.
Therefore $\U(\form{P})\subset\U(P+\form{P})$ strictly, which contradicts
\eqref{e:U(K*)=U(K+K*)} and concludes the proof.
\end{proof}

\begin{remark}
Unfortunately, the previous polytope $P$ does not provide us with a
counterexample for the non-increasing behavior of the isoperimetric
quotient function $I(\lambda)$. Thus, except for particular families of
convex bodies (see Section~\ref{s: quotient and deficit}) it is not known
yet whether the isoperimetric quotient function is non-increasing for an
arbitrary convex body.
\end{remark}

\section{Isoperimetric quotients and deficits}\label{s: quotient and deficit}

We start pointing out that unlike the authors in \cite{Domokos}, we are
going to consider the family of inner parallel bodies defined in the range
$-\inr(K;E)\leq\lambda\leq 0$, which will reverse the behavior of the
isoperimetric quotient function (in the original paper \cite{Domokos} the
range is $0\leq\lambda\leq\inr(K;\B_n)$ what makes the behavior of
$I(\lambda)$ non-increasing). Furthermore, since we will also work with
the isoperimetric deficit, we will consider the isoperimetric quotient in
the usual way, namely, $\S(K_{\lambda})^n/\vol(K_{\lambda})^{n-1}$, in
order to compare the behavior in both cases.

In this section we are going to obtain new results concerning the behavior
of the isoperimetric quotient (and also of the isoperimetric deficit)
under assumptions on the boundary of the involved convex bodies. The first
condition we can impose is, actually, the boundary condition necessary to
validate the proof of Theorem 1 in \cite{Domokos}, namely:

\begin{theorem}[{\cite[Theorem 1]{Domokos}} revised]\label{t:Domokos_ok}
Let $K\in\K^n_n$. If
\[
\U(\form{K_\lambda})=\U(K_\lambda+\form{K_\lambda})\quad\text{ for }
-\inr(K;\B_n)\leq \lambda\leq 0,
\]
then the isoperimetric quotient
$\S(K_{\lambda})^n/\vol(K_{\lambda})^{n-1}$ is non-increasing.
\end{theorem}

The proof of this result is exactly the proof of
\cite[Theorem~1]{Domokos}, where the use of \eqref{e: eq contraej} is
justified by assuming \eqref{e: cond extremos SY}.

Next we prove that under different conditions to \eqref{e: cond extremos
SY}, the isoperimetric quotient function is also non-decreasing. In fact,
we will get a more general result for all the quermassintegrals of a
convex body $K$ (relative to an arbitrary $E\in\K^n_n$), which we define
next.

Given $K\in\K^n$ and $E\in\K^n_n$, the so-called {\em relative Steiner
formula} states that the volume of the Minkowski addition $K+\mu E$,
$\mu\geq 0$, is a polynomial of degree $n$ in $\mu$,
\begin{equation*}\label{e:steiner_rel}
\vol(K+\lambda E)=\sum_{i=0}^n\binom{n}{i}\W_i(K;E)\lambda^i.
\end{equation*}
The coefficients $\W_i(K;E)$ are called {\em (relative) quermassintegrals}
of $K$, and they are just a special case of the more general {\em mixed
volumes}, for which we refer to \cite[s.~5.1]{Sch}. In particular, we have
$\W_0(K;E)= \vol(K)$ and $\W_n(K;E)=\vol(E)$. Moreover, if $E=\B_n$, the
polynomial in the right hand side becomes the classical {\em Steiner
polynomial}, see \cite{Ste40}, and $n\W_1(K;\B_n)=\S(K)$ is the usual
surface area of $K$.

Let $\W_i(\lambda):=\W_i(K_{\lambda};E)$ for $-\inr(K;E)\leq\lambda\leq
0$. From the concavity of the family of inner parallel bodies (see
\cite[Lemma~3.1.13]{Sch}) and the general Brunn-Minkowski theorem for
relative quermassintegrals (see e.g. \cite[Theorem~7.4.5]{Sch}), it is
obtained that
\[
'\W_i(\lambda)\geq\W_i'(\lambda)\geq(n-i)\W_{i+1}(\lambda)
\]
for $i=0,\dots,n-1$ and for $-\inr(K;E)\leq\lambda\leq 0$. Here $'\W_i$
and $\W_i'$ denote, respectively, the left and right derivatives of the
function $\W_i(\lambda)$, and for $\lambda=-\inr(K;E)$ (respectively,
$\lambda=0$) only the right (left) derivative is considered ($\W_i'$ will
also denote the full derivative of $\W_i$ when the function is
differentiable). In \cite{HCS} the following definition was introduced.
\begin{definition}\label{d:r_p}
Let $E\in\K^n_n$ and let $0\leq p\leq n-1$. A convex body $K\in\K^n$ {\em
belongs to the class $\r_p$} if, for all $0\leq i\leq p$ and for
$-\inr(K;E)\leq\lambda\leq 0$,
\begin{equation}\label{e:def_R_p}
'\W_i(\lambda)=\W_i'(\lambda)=(n-i)\W_{i+1}(\lambda).
\end{equation}
\end{definition}
Notice that the class $\r_p$ depends on the fixed convex body $E$.
Nevertheless, and for the sake of simplicity, we will also omit $E$ in the
notation.

Since the volume is always differentiable with respect to $\lambda$ and
$\vol'(\lambda)=n\W_1(\lambda)$ (see e.g. \cite{Bol43,Mat}), the class
$\r_0$ consists of all convex bodies, i.e., $\r_0=\K^n$. From the
definition we get $\r_p\supset\r_{p+1}$, $p=0,\dots,n-2$, and all these
inclusions are strict (particular tangential bodies show it; see
\cite{HCS}). The problem of determining the convex bodies belonging to the
class $\r_p$ was studied by Bol \cite{Bol43} and Hadwiger \cite{Ha55} in
the 3-dimensional case when $E=\B_n$. In \cite{HCS} and \cite{HCS_rbeta}
the general classes $\r_{n-1}$ and $\r_{n-2}$, respectively, were
characterized. The cases $p=1,\dots,n-3$ remains open.

Finally, we recall the following inequalities for quermassintegrals, which
can be deduced from the well-known Aleksandrov-Fenchel inequalities for
mixed volumes (see e.g. \cite[Sections 7.3 and 7.4]{Sch}). They motivate
and are also needed to prove our results. Let $K\in\K^n$ and $E\in\K^n_n$.
Then
\begin{equation}\label{e:AF}
\W_i(K;E)\W_j(K;E)\geq\W_k(K;E)\W_l(K;E),\;\; 0\leq l< i\leq j<k\leq n,
\end{equation}
and
\begin{equation}\label{e:AF2}
\W_j(K;E)^{n-i}\geq\W_i(K;E)^{n-j}\vol(E)^{j-i},\quad 0\leq i\leq j\leq n.
\end{equation}
We notice that the last inequality, for $E=\B_n$ and $i=0$, $j=1$, yields
the well-known isoperimetric inequality $\S(K)^n\geq
n^n\vol(\B_n)\vol(K)^{n-1}$.

\subsection{Non-decreasing isoperimetric quotients}

Inspired on the families of inequalities \eqref{e:AF2}, we consider the
isoperimetric quotient (up to the constant $\vol(E)^{j-i}$)
$\W_j(K;E)^{n-i}/\W_i(K;E)^{n-j}$ and study its behavior for the family of
inner parallel bodies.

We start proving that for convex bodies lying in the suitable class
$\r_p$, the above isoperimetric quotients for inner parallel bodies are
non-increasing in the range $-\inr(K;E)<\lambda\leq 0$.

\begin{proposition}\label{p: quermass inner}
Let $0\leq i<j<n$, and let $K\in\r_j$ and $E\in\K^n_n$. Then the
isoperimetric quotient function $\W_j(\lambda)^{n-i}/\W_i(\lambda)^{n-j}$
is non-increasing for $-\inr(K;E)<\lambda\leq 0$. In particular,
\[
\frac{\W_j(K_\lambda;E)^{n-i}}{\W_i(K_\lambda;E)^{n-j}}\geq\frac{\W_j(K;E)^{n-i}}{\W_i(K;E)^{n-j}}.
\]
\end{proposition}

\begin{proof}
We consider the function
\[
\phi(\lambda):=\frac{\W_j(\lambda)^{n-i}}{\W_i(\lambda)^{n-j}}\quad\text{
for }\;-\inr(K;E)<\lambda\leq 0.
\]
Taking derivatives with respect to $\lambda$, and since
$K\in\r_j\subset\r_i$ because $i<j$, we can use the relations
$\W_i'(\lambda)=(n-i)\W_{i+1}(\lambda)$ and
$\W_j'(\lambda)=(n-j)\W_{j+1}(\lambda)$ to get
\[
\begin{split}
\phi'(\lambda) & =\frac{\W_j(\lambda)^{n-i-1}}{\W_i(\lambda)^{n-j+1}}
    \Bigl[(n-i)\W_i(\lambda)\W_j'(\lambda)-(n-j)\W_j(\lambda)\W_i'(\lambda)\Bigr]\\
 & =\frac{(n-i)(n-j)\W_j(\lambda)^{n-i-1}}{\W_i(\lambda)^{n-j+1}}
    \Bigl[\W_i(\lambda)\W_{j+1}(\lambda)-\W_j(\lambda)\W_{i+1}(\lambda)\Bigr].
\end{split}
\]
The Aleksandrov-Fenchel inequalities \eqref{e:AF} yield
$\phi'(\lambda)\leq 0$, i.e., $\phi(\lambda)$ is non-increasing when
$-\inr(K;E)<\lambda\leq 0$.
\end{proof}

We note that if $K$ is a tangential body of $E$ then, since
$K_{\lambda}=(1+\lambda)K$ (Remark~\ref{r:tangential}) and the $i$-th
quermassintegral is homogeneous of degree $n-i$ in its first argument (see
e.g. \cite[Theorem~6.13]{Grub}), the isoperimetric quotient function
\[
\phi(\lambda)=\frac{\W_j(\lambda)^{n-i}}{\W_i(\lambda)^{n-j}}
=\frac{\W_j\bigl((1+\lambda)K;E\bigr)^{n-i}}{\W_i\bigl((1+\lambda)K;E\bigr)^{n-j}}
=\frac{\W_j(K;E)^{n-i}}{\W_i(K;E)^{n-j}}
\]
is constant in $-1<\lambda\leq 0$ for all $0\leq i<j<n$ (without
additional assumptions on the classes $\r_p$).

\begin{remark}
For $\lambda\geq 0$, denoting by $\W_i(\lambda)=\W_i(K+\lambda E)$, $0\leq
i\leq n-1$, one has $\W_i'(\lambda)=(n-i)\W_{i+1}(\lambda)$ for all
$i=0,\dots,n-1$ directly from the Steiner formula for quermassintegrals
(see \cite[(5.29) and p.~225]{Sch}); here, for $\lambda=0$ only the right
derivative is considered. This yields that the isoperimetric quotient
function $\phi$ defined for $\lambda\geq 0$ satisfies $\phi'(\lambda)\leq
0$ too, and thus, $\phi$ is non-increasing in the full range
$\bigl(-\inr(K;E),\infty\bigr)$. We observe that, when $\lambda\geq 0$, no
examples of constant $\phi$, apart from $K=E$, are known to the authors.
\end{remark}

The case $i=0$, $j=1$ (and $E=\B_n$) in Proposition~\ref{p: quermass
inner} provides us with an alternative result on the monotonicity of the
classical isoperimetric quotient with respect to the family of inner
parallel bodies, now under a different assumption on $K$ (cf.
Theorem~\ref{t:Domokos_ok}):

\begin{corollary}
Let $K\in\K^n_n$. If $K\in\r_1$ then the isoperimetric quotient
$\S(K_{\lambda})^n/\vol(K_{\lambda})^{n-1}$ is a non-increasing function
for $-\inr(K;\B_n)<\lambda\leq 0$. In particular,
\[
\frac{\S(K_\lambda)^n}{\vol(K_\lambda)^{n-1}}\geq\frac{\S(K)^n}{\vol(K)^{n-1}}.
\]
\end{corollary}

\subsection{Isoperimetric deficit}\label{ss: deficit}

Next we consider the isoperimetric deficit, instead of the quotient, of
the inequality \eqref{e:AF2}. As we will see, the behavior is the
opposite.

\begin{proposition}\label{quermass inner 2}
Let $0\leq i<j<n$, and let $K\in\r_i$ and $E\in\K^n_n$. Then the
isoperimetric deficit function
$\W_j(\lambda)^{n-i}-\W_i(\lambda)^{n-j}\vol(E)^{j-i}$ is non-decreasing
for $-\inr(K;E)<\lambda\leq 0$. In particular,
\[
\begin{split}
\W_j(K_\lambda;E)^{n-i} & -\W_i(K_\lambda;E)^{n-j}\vol(E)^{j-i}\\
 & \leq\W_j(K;E)^{n-i}-\W_i(K;E)^{n-j}\vol(E)^{j-i}.
\end{split}
\]

\end{proposition}

\begin{proof}
We consider the function
\[
\psi(\lambda)=\W_j(\lambda)^{n-i}-\W_i(\lambda)^{n-j}\vol(E)^{j-i}\quad\text{
for }-\inr(K;E)<\lambda\leq 0.
\]
Since $K\in\r_i$, we know that $\W_i(\lambda)$ is differentiable and
$\W_i'(\lambda)=(n-i)\W_{i+1}(\lambda)$; however, for $\W_j'(\lambda)$ we
can only take one-side derivatives, which satisfy
$'\W_j(\lambda)\geq\W_j'(\lambda)\geq(n-j)\W_{j+1}(\lambda)$. Thus, taking
the right derivative of $\psi(\lambda)$ with respect to $\lambda$ and
using the above relations, we obtain that
\[
\begin{split}
\psi'(\lambda) &
    =(n-i)\W_j(\lambda)^{n-i-1}\W_j'(\lambda)-(n-j)\W_i(\lambda)^{n-j-1}\W_i'(\lambda)\vol(E)^{j-i}\\
 & \geq (n-j)(n-i)\Bigl[
    \W_j(\lambda)^{n-i-1}\W_{j+1}(\lambda)-\W_i(\lambda)^{n-j-1}\W_{i+1}(\lambda)\vol(E)^{j-i}].
\end{split}
\]
Next we prove that
\begin{equation}\label{a}
\W_j(\lambda)^{n-i-1}\W_{j+1}(\lambda)\geq\W_i(\lambda)^{n-j-1}\W_{i+1}(\lambda)\vol(E)^{j-i}.
\end{equation}
Since $i<j$, we can use the relation
$\W_j(\lambda)^{n-i-1}\geq\W_{i+1}(\lambda)^{n-j}\vol(E)^{j-i-1}$
(cf.~\eqref{e:AF2}), and thus, in order to prove \eqref{a} it suffices to
show that
\begin{equation}\label{b}
\W_{i+1}(\lambda)^{n-j-1}\W_{j+1}(\lambda)\geq
\W_i(\lambda)^{n-j-1}\vol(E).
\end{equation}
If $j=n-1$, \eqref{b} holds trivially; so, we assume that $j\leq n-2$.

The family of inequalities given in \eqref{e:AF2} has a more general
version, namely, $\W_s^{k-l}(K;E)\geq\W_l^{k-s}(K;E)\W_k^{s-l}(K;E)$ for
$0\leq l\leq s\leq k\leq n$ (see e.g. \cite[(7.63)]{Sch}). Then, since
$i<i+1\leq n-j+i-1$, we can write
\[
\W_{i+1}(\lambda)^{n-j-1}\geq\W_i(\lambda)^{n-j-2}\W_{n-j+i-1}(\lambda),
\]
and hence, using also the Aleksandrov-Fenchel inequalities \eqref{e:AF} we
get
\[
\begin{split}
\W_{i+1}(\lambda)^{n-j-1}\W_{j+1}(\lambda) & \geq
    \W_i(\lambda)^{n-j-2}\W_{n-j+i-1}(\lambda)\W_{j+1}(\lambda)\\
 & \geq\W_i(\lambda)^{n-j-2}\W_i(\lambda)\vol(E)=\W_i(\lambda)^{n-j-1}\vol(E).
\end{split}
\]
It shows \eqref{b} and hence \eqref{a} holds.

So, we have that the right derivative $\psi'(\lambda)\geq 0$ for each
$\lambda\in\bigl(-\inr(K;E),0\bigr)$. Since $\psi$ is a continuous
function in the interval $\bigl[-\inr(K;E),0\bigr]$, \cite[Theorem~1]{MV}
yields that $\psi(\lambda)$ is non-decreasing when $-\inr(K;E)<\lambda\leq
0$, which conclude the proof.
\end{proof}

We point out that in Proposition \ref{p: quermass inner} we need that the
convex bodies lie in the class $\r_j$ whereas for Proposition
\ref{quermass inner 2} the assumption is weaker: the convex body has to
lie in $\r_i$, and $\r_j\subset \r_i$ because $i<j$. Therefore, in the
case of the classical isoperimetric deficit, i.e., $i=0$, $j=1$, since
$\r_0=\K^n$, no hypothesis is needed:

\begin{corollary}
For every $K\in\K^n$, the isoperimetric deficit
$\S(K_{\lambda})^n-n^n\vol(\B_n)\vol(K_{\lambda})^{n-1}$ is a
non-decreasing function for $-\inr(K;\B_n)<\lambda\leq 0$. In particular,
\[
\S(K_{\lambda})^n-n^n\vol(\B_n)\vol(K_{\lambda})^{n-1}\leq\S(K)^n-n^n\vol(\B_n)\vol(K)^{n-1}.
\]
\end{corollary}

Again if we define the isoperimetric deficit function for positive values
of $\lambda$, namely, $\psi(\lambda)=\W_j(K+\lambda
E;E)^{n-i}-\W_i(K+\lambda E;E)^{n-j}\vol(E)^{j-i}$, $\lambda\geq 0$, we
also get the same monotonicity, and thus $\psi(\lambda)$ is non-decreasing
in the full range $\bigl(-\inr(K;E),\infty\bigr)$.

\section{On the perimeter of inner parallel bodies}\label{s: perimeter}

In \cite[Theorem 1.2]{Larson} the author provides a lower bound for the
surface area of the inner parallel bodies of a convex body $K$ (with
respect to $\B_n$); following our notation, it is shown that
\begin{equation}\label{e:result_Larson}
\S(K_\lambda)\geq\left(1+\frac{\lambda}{\inr(K;\B_n)}\right)^{n-1}\S(K).
\end{equation}
For the proof of the above inequality, the author uses an auxiliary result
(\cite[Lemma~2.1]{Larson}) which states that, for $K\in\K^n_n$ and
$-\inr(K;\B_n)\leq\lambda \leq 0$, among all convex bodies $L\in\K^n$
satisfying that $L_\lambda=K_\lambda$, the set
$L=K_\lambda+\abs{\lambda}\form{K}$ has maximal surface area.

The proof of this lemma runs with the implicit assumption of a condition
closely related to \eqref{e: cond extremos SY}, namely,
\begin{equation}\label{e:cond_Larson}
\U(K+\form{K})=\U(K),
\end{equation}
which is necessary to have the following equality, last step in the proof:
\[
\bigcap_{u\in\U(K)}\Bigl\{x\in\R^n:\langle x,u\rangle\leq
h\bigl(K+\abs{\lambda}\form{K},u\bigr)\Bigr\}=K+\abs{\lambda}\form{K}
\]
(cf. \eqref{e:K_cap_U(K)}). Unfortunately, condition \eqref{e:cond_Larson}
is not satisfied for all convex bodies $K\in\K^n_n$: indeed, although
$\U(K)\subseteq \U(K+\form{K})$ always holds (Lemma~\ref{Collection of
properties} (iii)), the reverse inclusion needs not be true in general, as
the polytope $P$ given in \eqref{e: p} shows; we notice that
$\U(P)=\U(\form{P})\subset\U(P+\form{P})$ strictly.

The proof of Lemma 2.1 in \cite{Larson} actually yields that for
$K\in\K^n_n$ and $-\inr(K;\B_n)\leq \lambda \leq 0$, among all convex
bodies $L\in\K^n$ satisfying that $L_\lambda=K_\lambda$, exactly the set
$L=K(\lambda)$ (cf.~\eqref{e: family counterexample}) has maximal surface
area. For the sake of completeness we state it as a result.
\begin{lemma}[{\cite[Lemma 2.1]{Larson}} revised]
Let $K\in\K^n_n$ and $-\inr(K;\B_n)\leq \lambda \leq 0$. Among all convex
bodies $L\in\K^n$ satisfying that $L_\lambda=K_\lambda$, exactly the set
\[
K(\lambda)=\bigl\{x\in\R^n: \langle x,u\rangle \leq h(K,u)+\abs{\lambda}
h(E,u), u\in\U(K)\bigr\}
\]
has maximal surface area.
\end{lemma}

We conclude this note pointing out that, although the proof of Theorem~1.2
in \cite{Larson} is partially based on a not correct lemma, the result
itself is valid: the proof of \eqref{e:result_Larson} follows from
\cite[Lemma 2.9]{SYphd}, which states that
\[
K_\lambda\supseteq \left(1+\frac{\lambda}{\inr(K;\B_n)}\right)K,
\]
and the monotonicity and ($n-1$)-homogeneity of the surface area.

\medskip

\noindent {\bf Acknowledgements.} The authors would like to thank J. Yepes
Nicol\'as for pointing out to us the existence of results concerning the
behaviour of the isoperimetric quotient, in particular \cite{Domokos}.

\end{document}